\documentclass[11pt]{amsart}
\usepackage[left=2cm,top=2cm,right=3cm,nofoot]{geometry}
\usepackage{amsmath,mathtools}%
\usepackage{amsfonts}%
\usepackage{amssymb}%
\usepackage{graphicx}
\usepackage{esint}
\usepackage{hyperref}
\usepackage{enumitem}
\usepackage{accents}
\usepackage{etoolbox}
\patchcmd{\subsection}{-.5em}{.5em}{}{}
\newtheorem{theorem}{Theorem}[section]
\theoremstyle{plain}

\newtheorem{definition}[theorem]{Definition}

\newtheorem{lemma}[theorem]{Lemma}

\newtheorem{proposition}[theorem]{Proposition}
\newtheorem{remark}[theorem]{Remark}

\numberwithin{equation}{section}

\usepackage[left=2cm,top=2cm,right=3cm,nofoot]{geometry}
\usepackage{amsmath,mathtools}%
\usepackage{amsfonts}%
\usepackage{amssymb}%
\usepackage{graphicx}
\usepackage{esint}
\usepackage{hyperref}
\usepackage{enumitem}
\usepackage{accents}
\usepackage{etoolbox}

\theoremstyle{plain}

\usepackage{etoolbox}
\AtEndEnvironment{proof}{\setcounter{claim}{0}}

\newcommand{\scal}{\operatorname{scal}}

\newcommand{\g}{\mathfrak{g}}

\newcommand{\re}{\mathbb{R}}

\begin{document}
\title[Paneitz type equations]{Global bifurcation for Paneitz type equations and constant Q-curvature metrics}

\author{ Jurgen Julio-Batalla}
\address{ Universidad Industrial de Santander, Carrera 27 calle 9, 680002, Bucaramanga, Santander, Colombia}
\email{ jajuliob@uis.edu.co}
\thanks{  The first author was partially supported by project 3756 of Vicerrectoría de Investigación y Extensión of Universidad Industrial de Santander}

\author{Jimmy Petean}\thanks{The second author was supported by grant A1-S-45886 of Fondo Sectorial  SEP-CONACYT.}  
\address{Centro de Investigaci\'{o}n en Matem\'{a}ticas, CIMAT, Calle Jalisco s/n, 36023 Guanajuato, Guanajuato, M\'{e}xico}
\email{jimmy@cimat.mx}

\begin{abstract} We consider 
the Paneitz-type equation $\Delta^2 u -\alpha \Delta u +\beta (u-u^q  ) =0$ on a closed Riemannian manifold
$(M,g)$. We reduce the equation
to a fourth-order ordinary differential equation assuming that $(M,g)$ admits a proper isoparametric function. Assuming
that $\alpha$ and $\beta$ are positive and $\alpha^2 >4\beta$, we prove that the global nonconstant solutions of this
ordinary differential equation only has nondegenerate critical points. Applying global bifurcation theory we prove multiplicity results
for positive solutions of the equation. As an application and motivation we prove multiplicity results for conformal constant $Q$-curvature metrics.
For example, consider 
closed positive Einstein manifolds $(M^n ,g )$ and $(X^m , h)$ of dimensions $n,m \geq 3$. Assuming
that $M$ admits a proper isoparametric function (with a symmetry condition) we prove that as $\delta >0$ gets closer to 0, the number 
of constant $Q$-curvature metrics conformal to $\g_{\delta} = g+\delta h$ goes to infinity.
\end{abstract}
\maketitle

\section{Introduction}

We consider a closed Riemannian manifold $(M^n ,g)$, $n\geq 3$. Let $\Delta= \Delta_g =div_g \nabla$ be the
 (non-positive)
Laplace operator. If $n>4$ let
$p^* = (n+4)/(n-4)$, while for $n=3$ and $n=4$ we let $p¨* = \infty$.
The {\it Paneitz type equation}, with $\alpha , \beta \in \re_{>0}$, $q>1$, 

\begin{equation}\label{Paneitz}
\Delta^2 u -\alpha \Delta u +\beta (u-u^q  ) =0
\end{equation}

\noindent
is called critical if $q=p^*$, subcritical if  $q < p^*$ and supercritical if  $q > p^*$. In this article we will study multiplicity results
for positive solutions of this equation using global bifurcation techniques.

Paneitz type equations  appeared in Riemannian geometry in the study of the $Q$-curvature in dimensions greater than 4. 
Important results on the critical case were obtained by Z. Djadli, E. Hebey and M. Ledoux in \cite{Hebey}.

The $Q$-curvature of a Riemannian manifold was introduced by S. Paneitz \cite{Paneitz} and T. Branson \cite{Branson}
in their study of conformally invariant operators. It is defined for Riemannian manifolds $(\mathbf{M}, \g)$
of dimension $N \geq 3$ by

\begin{equation}\label{Qcurv}
Q_{\g}:=-\frac{1}{2(N-1)} \Delta_{\g} \scal_{\g}-\frac{2}{(N-2)^{2}}\left| \mathrm{Ric}_{\g}  \right|^{2}+\frac{N^{3}-4 N^{2}+16 N-16}{8(N-1)^{2}(N-2)^{2}} \scal_{\g}^{2}.
\end{equation}

\noindent
where $\scal_{\g}$ is
the scalar curvature and $\mathrm{Ric}_{\g}$ is the Ricci tensor.

The treatment in the case when the dimension is 3 or 4 is different, so we will assume in the rest of this discussion that the dimension 
is  $N\geq 5$. 

Closely related to the $Q$-curvature is the Paneitz operator $P_{\g}$, defined using a local $\g$-orthonormal frame 
$\left(e_{i}\right)_{i=1}^{n}$, as
\begin{equation}\label{opPan}
P_{\g} \psi:=\Delta_{\g}^{2} \psi+\frac{4}{N-2} \mathrm{div}_{\g} \left(    \mathrm{Ric}_{\g} \left(\nabla \psi, e_{i}\right) e_{i}\right)-\frac{N^{2}-4 N+8}{2(N-1)(N-2)} \mathrm{div}_{\g}\left(\scal_{\g} \nabla \psi\right)+\frac{N-4}{2} Q_{\g} \psi ,
\end{equation}
where  $\mathrm{div}_{\g}$ is the divergence.

 We  write a  conformal metric $\g \in\left[\g_{0}\right]$ as $\g=u^{\frac{4}{N-4}} \g_{0}$, for a positive function $u: \mathbf{M} \rightarrow \mathbb{R}$. 
The Paneitz operator is conformally invariant in the sense that for any $\psi: \mathbf{M} \rightarrow \mathbb{R}$,
\begin{equation*}
P_{\g} \psi=u^{-\frac{N+4}{N-4}} P_{\g_{0}}(u \psi).
\end{equation*}
It follows that we can express the $Q$-curvature of $\g$ in terms of 	$\g_{0}$ and $u$ by:
\begin{equation*}
Q_{\g}=\frac{2}{N-4} P_{\g}(1)=\frac{2}{N-4} u^{-\frac{N+4}{N-4}} P_{\g_{0}}(u).
\end{equation*}

\medskip

There has been great interest in studying the problem of finding metrics of constant $Q$-curvature in a given conformal class.
In the 4-dimensional case important results were obtained for instance by S-Y. A. Chang and P. C. Yang \cite{Chang}, 
by S. Brendle \cite{Brendle} and by Z. Djadli and A. Malchiodi \cite{Djadli}. But, as we mentioned before, we will focus in
this article  in the case when the dimension is $N\geq 5$.

Note that it follows from the previous comments that finding a conformal metric $\g=u^{\frac{4}{N-4}} \g_{0}$ with constant $Q$-curvature 
is equivalent to finding a positive solution $u$ of the fourth order {\it Paneitz-Branson equation}:
\begin{equation}\label{ConstantQcurv}
P_{\g_{0}} u=\lambda u^{\frac{N+4}{N-4}}, \quad \lambda \in \mathbb{R} .
\end{equation}

The problem of  the existence of metrics of constant $Q$-curvature in dimension $N\geq 5$ has been studied, for instance, by
J. Qing and D. Raske in \cite{Qing},  by E. Hebey and F.  Robert in \cite{Hebey}, by
M. J. Gursky and A. Malchiodi  in \cite{Malchiodi},
by
F. Hang and P. C. Yang  in \cite{Yang}. See also  \cite{Djadli, Esposito, Gursky, Hang, Robert}.  

Note that Einstein metrics have constant $Q$-curvature.
J\'{e}r\^{o}me V\'{e}tois \cite{Vetois} proved that if $\g$ is an Einstein metric with positive scalar curvature,  
different from  the constant curvature metric on the sphere, then it is the only metric of constant $Q$-curvature in its conformal class (up to multiplication by a constant). 
Conformal diffeomorphisms of the round sphere,  $(\mathbb{S}^N, \g_0 $), give  a non-compact family of solutions to the  Paneitz-Branson equation. 
But all the conformal metrics are actually isometric. The compactness of the space of positive solutions of the Paneitz-Branson equation 
was studied for instance in \cite{Hebey2, Li, YanYanLi, Malchiodi2}. High dimensional examples for which the space of solution is not compact
were constructed by J.  Wei and C.  Zhao in \cite{Wei}.

There are a few other cases where one can show the non-uniqueness of conformal metrics of constant $Q$-curvature.
One case,  is that of Riemannian products.
If $(M,g)$ and $(X,h)$ are closed Einstein manifolds with positive Einstein constant, of dimension at least $3$ then
for any $\delta >0$, the Riemannian product $(M\times X , g+\delta h)$ has constant positive $Q$-curvature 
and constant positive scalar curvature. 
It is easy to compute that as $\delta \rightarrow 0$ or $\delta \rightarrow \infty$ the {\it energy} of the
product metric $g+\delta h$ goes to   $ \infty$. 
We are in the conditions of a result proved in  \cite{Yang}, which says that in the conformal class of $g+\delta h$, there must be at least one other metric of constant $Q$-curvature, with uniformly bounded energy.

R. Bettiol, P. Piccione and Y. Sire \cite{Bettiol} considered appropriate Riemannian submersions and  reduce the equation to basic functions (functions which are constant along the fibers of the submersion). 
Using bifurcation theory for this reduced equation on the canonical variation obtained by varying the size of the fibers
(under certain conditions the total spaces of these variations have constant Q-curvature), they prove the existence of bifurcation instants, giving the existence of metrics in the family for  which there are other conformal metrics of constant $Q$-curvature.

\medskip

In the present article we will give multiplicity results for conformal constant $Q$-curvature metrics using global bifurcation theory.
These multiplicity results will come as a corollary of multiplicity results for the Paneitz-type Equation (\ref{Paneitz}) on a closed
Riemannian manifold $(M,g)$.. 

We will assume that $(M,g)$ admits a proper isoparametric function $F:M \rightarrow [t_0 , t_1 ]$. We recall that
$F$ being isoparametric means that both $\Delta F$ and $\| \nabla F \|^2$ are constant along the level sets of $F$.
And the isoparametric function $F$ is called proper if its level sets are connected. The only critical 
levels of $F$ are the minimum
and the maximum, which are submanifolds and are called the focal submanifolds. We will restrict Equation
(\ref{Paneitz}) to functions which are constant along the level sets of $F$.
Denoting by ${\bf d}$ the distance to
the focal submanifold $F^{-1} (t_0 )$, $D$ the distance between the focal submanifolds
$F^{-1} (t_0 ) $ and $F^{-1} (t_1 )$,   functions which are constant along the level sets of $F$ are expressed in the
form $u=\varphi \circ {\bf d}$ for some smooth function $\varphi : [0,D]  \rightarrow \re$. Details will be given 
in Section 2 , where we will show  that $u$ solves the equation (\ref{Paneitz}) if $\varphi$ solves the fourth order ordinary 
differential equation 

\begin{equation}\label{PaneitzODE}
\varphi  '''' + 2h \varphi  ''' +(2h' +h^2 -\alpha ) \varphi '' +(h''+hh ' -\alpha h )\varphi  ' + \beta (\varphi - \varphi^ q) =0.
\end{equation}

\noindent
where the function $h:(0,D) \rightarrow \re$ is the mean curvature of the (isoparametric) hypersurface ${\bf d}^{-1} (t)$. We
have that $\lim_{t \rightarrow 0} h(t) = +\infty$ and $\lim_{t \rightarrow D} h(t) = -\infty$. 
To obtain a well defined solution on $M$, $\varphi$ must be  defined in the interval $[0,D]$, $\varphi '  (0) = \varphi ''' (0) = \varphi ' (D) = \varphi ''' (D) =0$. 

Let $\lambda_i$, $0  > \lambda_1 >  \lambda_2 ...$ be the negative eigenvalues of the second order operator $L(v)= v'' + hv$
(which corresponds to the Laplacian on $F$-invariant functions).  In Section 5,  see Theorem 5.2 and Remark 5.3, we will prove:

\begin{theorem}Assume that the function $h$ is antisymmetric with respect to $D/2$. Assume that $\alpha, \beta >0$,
 and that $\alpha^2 >4\beta$. If $(1/2) (\alpha   - \sqrt{\alpha^2 + 4\beta  (q-1)} < - \lambda_i$  then Equation (\ref{PaneitzODE})
has at least $i$ different positive solutions: for each integer $k$, $1\leq k \leq i$ it has at least one positive solution with exactly $k+1$ critical points. 
\end{theorem}

The result will be obtained by  applying the global bifurcation theorem of P. Rabinowitz \cite{Rabinowitz}. We will consider 
positive, increasing, functions $\alpha (s)$, $\beta (s) : (0,\infty) \rightarrow (0,\infty)$, satisfying $\alpha^2 (s) > 4 \beta (s)$,
$\lim_{s\rightarrow 0} \alpha (s) =0$ and $lim_{s\rightarrow \infty} \alpha (s) = \infty$. 
We will also assume that the function $\phi (s) =(1/2) (\alpha   - \sqrt{\alpha^2 + 4\beta  (q-1)}$ es decreciente,
$\lim_{s\rightarrow \infty} =-\infty$. 
Note that for any positive
numbers $\alpha, \beta$ as in the theorem, there are functions $\alpha , \beta$ as before such that for some $s_0 \in
(0,\infty )$ we have that $\alpha (s_0 ) = \alpha$ and $\beta (s_0 ) = \beta$. Then we study bifurcation of the
equation $B(\varphi  ,s)=0$, where 

$$B(\varphi , s)= \varphi  '''' + 2h \varphi  ''' +(2h' +h^2 -\alpha (s)  ) \varphi '' +(h''+hh ' -\alpha (s) h )\varphi  ' + \beta (s)  (\varphi - \varphi^ q),$$

\noindent
from the family of trivial solutions $\{ (1,s) \}$. The linearized equation at the trivial solution $(1,s)$  is 

$$A(v)= D_{\varphi} B[(1,s)]  (v)=
L^2 v -\alpha (s)L v+\beta (s)(1-q)v=0  .$$

If $v_i$ is a $\lambda_i$-eigenvector of $L$ then $A(v_i ) = \left(  \lambda_i^2 -\alpha (s) \lambda_i +\beta (s) (1-q) \right) v_i$. It
follows that the linearized equation at $(1,s)$ has a nontrivial solution if and only if there exists
$i$ such that $ \lambda_i^2 -\alpha (s) \lambda_i = \beta (s) (q-1)$. Therefore the bifurcation points for the equation are
the points $(1,s_i )$ such that $(1/2) (\alpha (s_i )   - \sqrt{\alpha^2 (s_i )+ 4\beta (s_i ) (q-1)} =  \lambda_i$. The
bifurcation points are simple and using the classical result of M. G. Crandall and P. Rabinowitz \cite{Crandall} about bifurcation from simple
eigenvalues we have that the family of nontrivial solutions around $(1,s_i )$ is given by a path $x\mapsto (\varphi (x) ,s(x) )$,
with $\varphi (x) = 1+xv_i +o(x^2 )$.  Let $S_i$ be the connected component, in the closure of the space of nontrivial solutions, which
contains the path. To study  $S_i$ we will  prove in Section 4  that there are no positive solutions 
of Equation (\ref{Paneitz} for $\alpha, \beta $ close to zero:

\begin{theorem}
There exists $\alpha_0 >0$ such that  if $ \alpha$, $0<\alpha < \alpha_0$, $\alpha^2 -4\beta >0$ and $q<p^*$ then the 
only positive solution of Equation \ref{Paneitz} is the constant solution 1.

\end{theorem}

Next we prove that all the solutions in the connected components $S_i$  are positive functions. This follows from a maximum principle, a
non-negative solution of Equation (\ref{PaneitzODE}) is either the constant 0 or it is  strictly positive. Maximum principles are not available in general
for fourth order equations, but there are results regarding the Paneitz-Branson equation. M. Gursky and A. Malchiodi proved a maximum
principle for the Paneitz-Branson equation under certain positivity conditions \cite{Malchiodi}.
A maximum principle for the Paneitz-type Equation (\ref{Paneitz}) was applied by Z. Djadli, E. Hebey and M. Ledoux in \cite{Hebey}
by expressing the fourth order operator with constant coefficients 
as the product of two appropriate second order operators. For the equation (\ref{PaneitzODE}) the
argument works in the same way, we will do it in Section 3. Then we have to prove that the connected components coming
from different bifurcation points do not intersect. For this we will prove that the number of critical points remains constant in 
any of the connected components. To prove this we will prove in Section 3 that nontrivial positive solutions of Equation
(\ref{PaneitzODE}) only have nondegenerate critical points. This again is not true in general for fourth order equations and it will be proved in Section 3
assuming that $h$ is antisymmetric around $D/2$.

Then it follows from the global bifurcation theorem of P. Rabinowiz that the connected component is not compact. For $q<p^*$ it is proved
in Section 4 that
the space of solutions of the equation with bounds on $\alpha, \beta$ is compact. Then the multiplicity result  follows from standard arguments.

\medskip

From the discussion above we can obtain multiplicity results for constant $Q$-curvature metrics. We will discuss now the case of 
Riemannian products of positive Einstein metrics.

Consider the Riemannian product $(M^n \times X^m, \g =g+h)$, where $g$ is Einstein with positive Einstein constant $\Lambda_0$ and $h$ is
Einstein with positive Einstein constant $\Lambda_1$.  Let $N=n+m$. Then $\g = g+h$ has constant $Q$-curvature 

$$ Q_{\g}= \frac{-2}{(N-2)^2} \left( n\Lambda_0^2  + m\Lambda_1^2  \right)  + 
\frac{N^3 -4N^2 +16N-16}{8(N-1)^2 (N-2)^2} \left( n\Lambda_0 + m \Lambda_1 \right)^2 .$$

We have that $Q_{\g} >0$ if $m,n \geq 3$ (see the computation in \cite[Lemma 2.2]{Alarcon}). 

In this case the Paneitz-Branson Equation (\ref{ConstantQcurv}) for a function on $M$  is

\begin{equation}\label{PaneitzTypeEquation}
\Delta_g^2 u -\alpha \Delta_g u +\beta (u-u^{\frac{N+4}{N-4}} ) =0
\end{equation}

\noindent
with $\alpha =\frac{N^2-4N+8}{2(N-1)(N-2)} (n\Lambda_0 + m\Lambda_1 ) -  \frac{4}{N-2} \Lambda_0$ and $\beta =\frac{N-4}{2} Q_{\g}$.
Note that $\alpha , \beta >0$. We can also check that $\alpha^2 -4\beta >0$ if $\Lambda_1 \geq \Lambda_0$. This is an elementary 
computation: we give the details in the Appendix at the end of the article, for the convenience of the reader.

It is needed that at least one of the factors admits a proper isoparametric function. The simplest situation is when the Riemannian manifolds
admits an isometric cohomogeneity one action. We require that the quotient space of the action is an interval $[0,D]$ (it is a circle the
manifold will have infinite fundamental group, and could not admit a positive Ricci curvature metric). The antisymmetric condition on $h$ 
is fulfilled if there is an extra isometry of $(M,g)$ which sends the t fiber to the D-t fiber, for every, $t \in [0,D/2]$. There are certainly 
plenty of examples. We will focus on isoparametric functions on the sphere. In this case one has many isometric actions as required
and also isoparametric functions which do not come from cohomogeneity one isometric actions. We will give more details in Section 2. 
Then we consider the family $s\in (1,\infty ) \mapsto \g_s = g+s^{-1} h$. Assuming we have an appropriate isoparametric function on 
$(M,g)$  we reduce the Paneitz-Branson Equation (\ref{ConstantQcurv}) for $\g_s$ to the ordinary differential equation
(\ref{PaneitzODE}). Then Theorem 1.1 implies that the number of  constant $Q$-curvature metrics conformal to $\g_s$ goes to
infinity as $s\rightarrow \infty$.

\medskip

To give a feeling of the multiplicity result obtained for conformal constant $Q$-curvature metrics we
will focus now  on the product of $({\mathbb S}^3 \times {\mathbb S}^3, g_0 + g_0 ) $, 
with $g_0$ the round metric with curvature 1. 
On ${\mathbb S}^3$ there are two (families of) isoparametric functions, corresponding to isometric actions
of $O(3)$ and of $O(2) \times O(2)$. ${\mathbb S}^3$ is of course Einstein with Einstein constant 2. The product $g_0 + g_0$ does not admit any
other conformal constant $Q$-curvature metric by the result of J. Vetois \cite{Vetois} mentioned above. We consider the products 
$g_0 + s^{-1}  g_0$, with $s \in (1, \infty )$. Then $\Lambda_1 = s \Lambda_0$ and we are in the situation described above. We
will build solutions which depend on the first factor (which would be the {\it large} factor). Then $\alpha (s)= 3s+1$ and
$\beta (s) = -(3/2)(1+s^2 ) + (171/100) (1+s)^2$. 
In the case of ${\mathbb S}^3$ the Laplace eigenvalues are $\mu_i =-i(i+2)$. 
We consider the  values $s_i$ such that $-i(i+2)= (1/2) \left( \alpha (s_i) - \sqrt{\alpha^2 (s_i ) +16 \beta (s_i ) (q-1)} \right)$ 
($s_i$ of course can be obtained explicitly).
For an $O(3)$-invariant isoparametric function the invariant eigenvalues are the same, $\lambda_i = \mu_i$. For an $O(2) \times O(2)$-invariant
isoparametric function the invariant eigenfunctions are $\lambda_i = \mu_{2i}$. Then Theorem 1.1 implies that for $s\in (s_{2i} ,s_{2i+1}]$ there
are $3i$ metrics of constant $Q$-curvature metrics conformal to $\g_s$; $2i$ of the corresponding conformal factors are $O(3)$-invariant and
$i$ of them are $O(2) \times O(2)$-invariant. For $s\in (s_{2i +1} ,s_{2i+2}]$ there are at least $3i+1$ metrics of
constant $Q$-curvature conformal to $\g_s$.

\section{Paneitz-type equations and isoparametric functions}

\noindent
Let ($M^n,g$) be a closed connected Riemannian manifold of dimension $n$.
Recall that a function $F:M\rightarrow[t_0,t_1]$  is called  {\it isoparametric} if  $|\nabla F|^2=b(F)$, $\Delta F=a(F)$ for 
some smooth functions 
$a,b$ defined on an interval containing $[t_0 , t_1 ]$. 
Isoparametric functions on general Riemannian manifolds were first considered by Q-M. Wang in
\cite{Wang}, following the classical theory of Cartan, Segre and others in the case of space forms
(see for instance \cite{Cartan, FKM1981, Civita, m1, m2, OT1975, Segre}). 

We will recall now a few known results about isoparametric functions. We are interested in writing the
Paneitz-type Equation (\ref{Paneitz}) for functions which are constant along the level sets of an isoparametric function,
which amounts to compute the Laplacian on such functions. Explicit details were shown in \cite[Section 2]{AJJ},
where the reader can find appropriate references. 

\medskip

Given an isoparametric function, with the corresponding functions $a$ and $b$,
it is known that the only zeroes of the function  $b:[t_0,t_1]\rightarrow [0,\infty)$ are $t_0$ and $t_1$,  which 
means that the only critical values of $F$ are its minimum and its
maximum (see \cite{Wang}). It is also proved in \cite{Wang} that the critical 
level sets, $M_0 = F^{-1} (t_0 )$ and
$M_1 = F^{-1} (t_1 )$ are smooth submanifolds; they are called the {\it focal submanifolds}
of $F$. We let $d_i$ be  the dimension of $M_i$. We will assume that
$d_0 , d_1 \leq n-2$. In this case the isoparametric function $F$ is called {\it proper}  and all the level sets
$M_t = F^{-1} (t)$ are connected (as proved in \cite[Proposition 2.1]{Tang}). It is also known that 
for any $t\in (t_0 , t_1 )$ the
level set $M_t$ is a tube  around any of the focal submanifolds: all points in $M_t$ are at 
the same distance to $M_0$ (and at the same distance to $M_1$).  

Let $D= d_g (M_0 , M_1 )$, and  ${\bf d} : M \rightarrow [0,D]$,
${\bf d} (x) =d_g(M_0 ,x)$. The function ${\bf d}$ is continuous, and smooth away from $M_0$ and $M_1$.




It is easy to see (see \cite[Section 2]{AJJ} for some details) that for any $c,d \in [t_0 , t_1 ], c<d,$

\begin{equation*}
d_g(M_c,M_d)=\int_{c}^{d}\dfrac{1}{\sqrt{b(t)}} dt,
\end{equation*}

\noindent
and, in particular,

\begin{equation*}
D=\int_{t_0}^{t_1}\dfrac{1}{\sqrt{b(t)}} dt.
\end{equation*}

\medskip

We will consider functions which are constant on the level sets of $F$:

\begin{definition}
A function $u:M\rightarrow \mathbb{R}$ is called $F$-invariant if $u(x)=\phi({\bf d}(x))$ for some  function $\phi:[0,D ]\rightarrow \mathbb{R}$.
\end{definition}

An $F$-invariant function $u$ is obviously determined by the corresponding function $\phi$. 
The function $u$ is smooth on $M- M_0  \cup M_1$ if and only if the corresponding function $\phi$ is
smooth on $(0,D)$; and it can be extended smoothly to $M$ if $\phi$ can be extended smoothly to 
$[0,D]$ with all the odd derivatives at $0$ and $D$ vanishing (in the same way as smooth radial functions on
Euclidean space are identified with smooth functions on $[0,\infty )$ with odd derivatives vanishing at 0).

We will denote by $\mathcal{B}_2 = \{ \phi \in C^{2,\alpha} ([0,D] ): \phi '(0) =0 = \phi '(D) \}$and
by $\mathcal{B}_4 = \{ \phi \in C^{4,\alpha} ([0,D] ): \phi '(0)  =  \phi ''' (0)=0 = \phi '(D)  = \phi '''(D)\}$.

\begin{lemma}
If we denote by
$C^{2,\alpha}_F (M)$ the set of $C^{2,\alpha}$ functions on $M$ which are $F$-invariant and by
$C^{4,\alpha}_F (M)$ the set of $C^{4,\alpha}$ functions on $M$ which are $F$-invariant, then the
application $\phi \mapsto u(x)=\phi({\bf d}(x))$ identifies $\mathcal{B}_2$ with $C^{2,\alpha}_F (M)$
and $\mathcal{B}_4$ with $C^{4,\alpha}_F (M)$.
\end{lemma}

The lemma is well known, it follows from the previous comments. Some details are given in \cite[Section 2]{AJJ}.

For $t\in (0,D)$ we let $h(t)$ be the mean curvature of the hypersurface ${\bf d}^{-1} (t)$. Then for any
$x\in M- M_0 \cup M_1$ we have the expression 

$$h({\bf d} (x)) =\frac{1}{2\sqrt{b}}(-b'+2a) (f(x) ) = \Delta ({\bf d} (x) $$

\noindent
(see for instance  \cite[page 165]{Tang}, \cite[Section 2]{AJJ}, for details).

Then for $u= \phi \circ {\bf d}$ and $x \in M-M_0 \cup M_1$ we have:

\begin{equation}
\Delta (u) (x)= \Delta(\phi({\bf d}(x) = \phi^{''}  (  {\bf d}(x) ) +  \phi'  ({\bf d}(x) )  h({\bf d} (x) ) = (\phi '' + \phi ' h)({\bf d} (x) ).
\end{equation}

\medskip

The function $h$ is clearly smooth in $(0,D)$. The asymptotic behaviour of the mean curvature close to focal submanifolds
is also easy to check since the hypersurfaces are the normal geodesic spheres around the focal submanifolds. 
In \cite[Corollary 2.2]{Ge}, for instance, the authors study the power series expansion for $h$ close to a focal submanifold.
We have:

\begin{lemma}\label{mean-curvature}

$$\lim_{t\rightarrow 0}  t   \ h(t) = n-d_0 -1 \ \ \ ,  \ \ \ \ \lim_{t\rightarrow D} (t-D )   \ h(t) = n-d_1 -1. $$

\end{lemma}

Then  taking the limit we see that if if $x\in M_0 $

\begin{equation}
\Delta u (x) =(n-d_0 )  \phi'' (0) 
\end{equation}

\noindent
and if $x\in M_1 $

\begin{equation}
\Delta u (x) =(n-d_1 )  \phi'' (D ) .
\end{equation}

\medskip

Note that  we have

\begin{equation}
\Delta^2 (u) = (\phi ''''  + 2 \phi ''' h + 2 \phi '' h' + \phi ' h'' + \phi '' h^2 + \phi ' h' h ) \circ {\bf d}.
\end{equation}

Therefore,

\begin{lemma}Let $u \in C^{4,\alpha}_F (M)$, $ u(x)=\varphi({\bf d}(x))$, with $\varphi \in \mathcal{B}_4$. Then $u$ is a solution of  the
Paneitz-type Equation
(\ref{Paneitz}) 

$$\Delta^2 u -\alpha \Delta u +\beta (u-u^q  ) =0$$

\noindent
if and only if the function $\phi$ satisfies the Equation (\ref{PaneitzODE})

$$\varphi  '''' + 2h \varphi  ''' +(2h' +h^2 -\alpha ) \varphi '' +(h''+hh ' -\alpha h )\varphi  ' + \beta (\varphi - \varphi^ q) =0$$

\noindent 
on $[0,D ]$.

\end{lemma}

\bigskip

We will assume that the function $h$ is antisymmetric around $D/2$, $h(D/2 -t ) = - h(D/2 + t)$. This is the case for instance
if there is an isometry ${\mathcal I}$ of $(M,g)$ which is a reflection around ${\bf d}^{-1} (D/2)$; it
sends ${\bf d}^{-1} (D/2 +t )$ to ${\bf d}^{-1} (D/2 -t )$. 

\medskip

The classical and important case of the round sphere can be treated explicitly.
 A Cartan-M\"{u}nzner  polynomial $P$
of degree $k$ on $\re^{n+1}$ is a polynomial, of degree $k$, which verifies

$$ \langle \nabla P , \nabla P \rangle = k^2 \| x \|^{2k -2} $$

$$ \Delta P = \frac{1}{2} c k^2 \| x \|^{k -2} ,$$

\noindent
which are called the Cartan-M\"{u}nzner equations. 
Then the restriction to the sphere, $p=P_{|_{S^n}}$, verifies the isoparametric conditions
$\|  \nabla p \|^2 = b\circ p$, $\Delta p = a \circ p$, with 
$b(t)= k^2 (1-t^2 )$, $a(t) =-k(n+k-1)t+ (1/2)ck^2$. 
It follows that $p:S^n \rightarrow [0,1]$.
The constant $c$ depends on the geometry,
the principal curvatures, of the level surfaces $p^{-1} (t)$, $t\in (0,1)$.

Every isoparametric function on the sphere $S^n$ is
given (up to reparametrization) by the restriction of a Cartan-M\"{u}nzner polynomial on $\re^{n+1}$.

Note that in this case

$$\frac{-b' +2a}{2\sqrt{b}} (t) = \frac{-2k(n-1)t +ck^2}{2k\sqrt{1-t^2}} .$$

We obtain that

$${\bf d} (x) = \int_{-1}^{f(x)} \frac{1}{\sqrt{b(s)}} ds = \frac{1}{k}  \left(  \arcsin (f(x)) + \frac{\pi}{2} \right),  \ \  \  D= \frac{\pi}{k}$$

and then, with $t={\bf d}(x)$  so $f(x)= \sin (kt-\pi /2)  =-\cos (kt)$, we have

$$h(t)=\frac{-b' +2a}{2\sqrt{b}} (- \cos (kt))= (n-1)\frac{\cos (kt)}{\sin (kt)} +\frac{ck}{2\sin (kt)}.$$ 

\noindent
which is invariant around $D/2=\pi /2k$ if and only if $c=0$.

The classification of isoparametric functions on spheres was only recently completed by Q-S Chi in \cite{Chi}, concluding the work
started by E. Cartan and others around 1930. 
It is known, \cite{m1, m2}, that $k$ can only take the values 1, 2, 3, 4 and 6. 
The case $k=1$ corresponds to linear functions, which satisfy the 
Cartan-M\"{u}nzner equations with $c=0$. For all Cartan-M\"{u}nzner polynomials of degree 3 and 6 we also have $c=0$. The case
$k=2$ corresponds to polynomials invariant by an action of $O(l)\times O(j)$, with $l+j = n$: in this case 
$c=0$ if and only if $j=l$. In the case $k=4$ there are also examples with $c=0$, and examples with $c\neq 0$. Detailed
discussions on isoparametric functions on spheres can be found for instance in \cite{Henry}.

\bigskip

We consider now the case of the 3-sphere, $S^3$. Recall that the eigenvalues of the positive Laplacian on $S^3$ are
$\mu_i = -i(i+2)$, $i=0,1,\dots$. 

In this case there are 2 types of isoparametric functions. The first case are 
the isoparametric functions invariant by an action of $O(3)$, fixing an axis in $\re^4$. These correspond to the 
Cartan-M\"{u}nzner polynomial of degree one, the linear functions. The eigenvalues of the Laplacian on functions 
invariant in this case are the same as the eigenvalues of the Laplacian, $\lambda_i = \mu_i$. For the  
Paneitz-type Equation \ref{Paneitz} for invariant functions, the function $h$ is $h(t) = 2\frac{\cos t}{\sin t}$, and
the equation is defined on $[0,\pi ]$. 

The second case are functions invariant by an isometric action by $O(2) \times O(2)$ on $\re^4$. In this case $k=2$, the
corresponding Cartan-M\"{u}nzner polynomial is $P(x_1 ,x_2 ,x_3 , x_4 ) = x_1^2 +x_2^2 -x_3^2 -x_4^2$. The eigenvalues
of the restricted Laplacian are $\lambda_i =\mu_{2i}$. The corresponding equation is defined in the interval $[0,\pi /2]$, with
$h(t) =2 \frac{\cos 2t}{\sin 2t}$.

\section{Results on the ordinary differential equation}

In this section we will study ordinary differential equation \ref{PaneitzODE}, where we assume
that $\alpha , \beta >0$ and $\alpha^2  \geq 4\beta$, $q>1$. The function $h$ will be defined in the open interval $(0,D)$.
We will  assume that $h$ is decreasing, antisymmetric around the middle point $d=D/2$. Namely, $h(d) =0$, 
$h(d+s) =-h(d-s)$. This implies in particular that also $h''(d)=0$. And in the limiting points $0$ and $D$ we assume that
 $\lim_{t\rightarrow 0} th(t) =j=-\lim_{t\rightarrow D} (D-t)h(t)$,
for some positive integer $j$.

We begin with the following elementary observation:

\begin{lemma}
Under the previous hypothesis if $\varphi$ solves Equation \ref{PaneitzODE} then $\phi (t) = \varphi (D-t)$ also solves Equation\ref{PaneitzODE}.
\end{lemma}

\vspace{.4cm}

We write \ref{PaneitzODE} as a first order system:

\vspace{.5cm}

\hspace{5cm}$v_0 ' =v_1$

\hspace{5cm}$v_1 ' = v_2 -hv_1$

\hspace{5cm}$v_2 ' =v_3 + \alpha v_1$

\hspace{5cm}$v_3 ' =\beta (v^q_0 -v_0 )  -hv_3 .$

\vspace{.4cm}

In this system $v_0$ is the solution of  Equation \ref{PaneitzODE}.  
Note that $v_2 = v_0 '' +h v_0 '$,  and $v_3 =v_1 '' + h ' v_1 + h v_1 ' - \alpha v_1  = v_0 '''  + h v_0 '' +h ' v_0 ' 
-\alpha v_0 '$. The system is similar to  \cite[(2.2)]{Lazzo}, but it is not quasimonotone.

\begin{remark}\label{Decomposition}
We assume that $\alpha, \beta >0$ and that $\alpha^2 \geq 4\beta$. This implies that there exists $c, d>0$ such that
$c.d=\beta$ and $c+d =\alpha$. Explicitly, $c=(1/2) (\alpha +\sqrt{\alpha^2 -4\beta})$ and $d=(1/2) (\alpha - \sqrt{\alpha^2 -4\beta})$.
\end{remark}

\begin{remark}\label{Basic}
It follows easily from the system that a solution which is continuous and defined at $0$ must verify that $v_1 (0) = v_3 (0) =0$. Also
by the conditions we are assuming on $h$ we have that $\lim_{t\rightarrow 0} v_1 (t) h(t) = j v_1 '(0)$. It follows that 
$v_2 (0) = (1+j) v_1 ' (0)$. Similarly  $(1+j) v_3 ' (0) =
\beta  (v^q_0 -v_0 ) (0)$. In the same way we have that $v_2 (D) = (1+j)v_1 ' (D)$ and $(1+j) v_3 ' (D) = \beta (v_0^q -v_0 )(D)$.
\end{remark}

\begin{remark}
Given any pair $(a,b) \in \re^2$ there exists a unique solution of the system with initial conditions $v_0 (0) = a$,
$v_2 (0) = b$, $v_1 (0) =v_3 (0) =0$. The solutions are $C^1$,  and then it follows that 
$v_i$ is $C^{5-i}$. In particular $v_0$ is $C^5$, and a standard solution of Equation \ref{PaneitzODE}. The solutions
depend $C^1$-continuously on the initial parameters $(a,b)$.
\end{remark}

\begin{definition}\label{Global}
A solution $(v_0 , v_1 , v_2 , v_3 )$  of the system  is called global if it is defined on the whole interval $[0,D]$. In particular we must have
$v_1 (0)=v_3 (0) =v_1 (D) = v_3 (D)=0$
\end{definition}

Let  $H(t) = e^{\int_d^t h(s) dx}$. Note that $H(0)=0$, $H(D)=0$, $H(d)=1$ , $H(t)>0  \ \forall t\in (0,d)$ and $H'(t)=h(t) H(t)\ \  \forall t\in (0,D)$.

Note also that $(Hv_3 )' = H(v_3 ' + hv_3 )$, $ ( Hv_1 )' = H (v_1 ' + h v_1 )$. 

Then we have that if $(v_0 , v_1 , v_2 , v_3 )$ is a global solution of the system then 

\begin{equation}
\beta \int_0^D (v_0^q (t) -v_0 (t) H(t) dt = \int_0^D (v_3 ' (t) +h(t) v_3 (t) )H(t) dt = \int_0^D  (Hv_3 )' (t) dt =0.
\end{equation}

It follows that

\begin{lemma}\label{AboveBelowOne}
Assume that $(v_0 , v_1 , v_2 , v_3 )$ is a global  solution of the system. If $v_0 \geq 1$ then $v_0 \equiv 1$. If $0\leq v_0 \leq 1$ then $v_0 \equiv 0$
or $v_0 \equiv 1$. 

\end{lemma}

\begin{lemma}\label{MaximumPrinciple} Assume that $(v_0 , v_1 , v_2 , v_3 )$ is a global solution of the system and that $v_0 \geq 0$.
Then $v_0 \equiv 0$ or $v_0 >0$.
\end{lemma}

\begin{proof} As mentioned in Remark \ref{Decomposition} there exist $c, d >0$ such that 
$c.d =\beta$ and $c+d =\alpha$. Let $w_2 = v_2 -cv_0$. Then $w_2 '' + h w_2 ' -d w_2 =(v_3 + \alpha v_1 )' +h (v_3 +\alpha v_1) -cv_1 '-chv_1 -dv_2 +cdv_0
=\beta (v_0^q -v_0) -hv_3 +\alpha (v_2 -hv_1 ) +hv_3 +\alpha h v_1 -c(v_2 -hv_1 )-chv_1 -dv_2 +\beta v_0 =
\beta v_0^q \geq 0$.
This implies that $w_2 \leq 0$, since it cannot have a strictly positive maximum. 

Assume that  at some point $t_0 \in [0,D]$ we have that $v_0 (t_0 )=0$. 
Since $0$ is a local minimum of $v_0$ we must also have that $v_1 (t_0 )=0$ and $v_1 '(t_0 ) \geq 0$.
This implies that $v_2 (t_0 ) \geq 0$, and since $w_2 (t_0 )\leq 0$ we must have that $v_2 (t_0 )=0$.
Then $w_2 (t_0 )=0$ and $w_2 ' (t_0 )=0$ since $t_0$ is a maximum for $w_2$. Then $v_2 ' (t_0 )=0$.
Therefore $v_3 (t_0 )=0$. Then by the uniqueness of
solutions, $v_0 \equiv 0$.

\end{proof}

\begin{lemma}\label{Negative}
Let $(v_0 , v_1 , v_2 , v_3 )$ be a solution of the system. If for some $t_0 \in (0,D)$ we have that
$v_0  (t_0 )<1$ and $v_1 (t_0 ) , v_2 (t_0 ), v_3 (t_0 ) < 0$ then $(v_0 , v_1 , v_2 , v_3 )$ is not a global nonnegative solution.
\end{lemma}

\begin{proof} Note that by the previous lemma if $(v_0 , v_1 , v_2 , v_3 )$ were a global nonnegative solution we would have that $v_0 >0$.
If it were a global solution we would have that $v_1 (D)=0$. Therefore there must be a first value
$t_1 > t_0$ such that one of the inequalities in the statement of the lemma
becomes an equality. Note that $v_0 (t_1 ) <1$ since 
$v_0 (t_0 ) <1$ and $v_0 ' =v_1 <0$ in $[t_0 , t_1 )$.  Since $v_1 , v_3 <0$ in $[t_0 , t_1 )$ it
also follows that $v_2 (t_1 )<0$. Note that $(Hv_1 ) ' = H (v_1 ' + hv_1 )= Hv_2 .$ Since
$v_2 <0$ in $[t_0 , t_1 )$ and $H(t_0 ) v_1 (t_0 ) <0$ it follows that $v_1 (t_1 )<0$. In
particular we must have $t_1 <D$. Finally $(H .v_3 )' =H (v_3 ' + hv_3 ) = H \beta (v_0^q -v_0 ) <0$
in $[t_0 , t_1 ]$ which implies that $v_3 (t_1 ) <0$. This is a contradiction with the 
definition of $t_1$ and therefore $(v_0 , v_1 , v_2 , v_3 )$ cannot be global.

\end{proof}

Similarly we have:

\begin{lemma}\label{Positive}
Let $(v_0 , v_1 , v_2 , v_3 )$ be a solution of the system. If for some $t_0 \in (0,D)$ we have that
$1<v_0 (t_0 ) $ and $v_1 (t_0 ) , v_2 (t_0 ), v_3 (t_0 )> 0$ then $(v_0 , v_1 , v_2 , v_3 )$ is not a global solution.
\end{lemma}

\begin{proof}
If it were a global solution we would have that $v_1 (D)=0$. Therefore, as in the previous lemma, 
there must be a first value
$t_1 > t_0$ such that one of the inequalities in the statement becomes an equality. But $v_0 (t_1 ) >1$ since 
$v_0 (t_0 ) >1$ and $v_0 ' =v_1 >0$ in $[t_0 , t_1 )$.  Since $v_1 , v_3 >0$ in $[t_0 , t_1 )$ it
also follows that $v_2 ' (t ) >0$ for all $t \in [t_0 , t_1 )$ and therefore
$v_2 (t)>0$ in $[t_0 , t_1 ]$. Since $(Hv_1 ) ' = H (v_1 ' + hv_1 )= Hv_2 $ and
$v_2 >0$ in $[t_0 , t_1 )$, $H((t_0 ) v_1 (t_0 ) >0$ it follows that $v_1 (t_1 )>0$. In
particular we must have $t_1 <D$. Finally $(H .v_3 )' =H (v_3 ' + hv_3 ) = H \beta (v_0^q -v_0 ) >0$
in $[t_0 , t_1 ]$ which implies that $v_3 (t_1 ) >0$. This is a contradiction with the 
definition of $t_1$ and therefore $(v_0 , v_1 , v_2 , v_3 )$ is not a global solution.

\end{proof}

\begin{lemma}\label{One}
A solution must verify $v_1 (0) = v_3 (0) =0$. Assume that $v_0 (0)=1$. Then $v_2 (0) =0$ and  $v_0 \equiv 1$, 
or the solution is not global and nonnegative.

\end{lemma}

\begin{proof} If $v_2 (0)=0$ we have that $v_0 \equiv 1$ by the uniqueness of solutions. 

If $v_2 (0)>0$ then there exists an interval $[0,\delta )$ such that $v_2 (t) =(hv_1 + v_1 ' )(t) >0$ for
all $t \in [0,\delta )$. Then $(Hv_1 )' (t) =(t) H(hv_1 + v_1 ' ) (t)>0$  for all $t \in  (0,\delta )$. Then $v_1 (t) >0 $ for all $t\in (0,\delta )$.
Then $v_0 (t) >1$ for all $t\in (0,\delta )$. The $(H.v_3 ) ' (t) = H(v_3 ' + hv_3 ) (t) =H \beta (v_0^q -v_0 )(t) >0$
for all $t\in (0,\delta )$. It follows that $v_3 (t) >0$ for all $t\in (0,\delta )$.
Then it follows from Lemma \ref{Positive} that the solution is not global.

Similarly, if $v_2 (0)<0$ we prove that there exists $\delta <0$ such that for all $t\in (0,\delta )$ we have that
$v_0 (t) <1$ and $v_1 (t), v_2 (t) , v_3 (t) <0$. And it follows from Lemma \ref{Negative} 
that the solution is not global and nonnegative.

\end{proof}

\begin{lemma}\label{AtZero} Let $(v_0 , v_1 , v_2 , v_3 )$ be a global non constant solution of the system with $v_0 >0$. Then
$v_0 (0) >1$ and $v_2 (0) <0$ or $v_0 (0) <1$ and $v_2 (0) >0$.

\end{lemma}

\begin{proof}
If $v_0 (0) >1$ there exists $\delta >0$ such that on the interval $[0,\delta]$ we have that
$v_0 >1$. Then $(Hv_3 )' (t) = H\beta (v_0^q -v_0 )(t) >0$ for all $t \in (0,\delta]$. It follows that 
$v_3 (t)>0$ for all $t\in (0,\delta ]$.

If $v_2 (0) >0$ then by taking $\delta$ smaller if necessary we can assume that  $v_2 (t) >0$ for all $t\in (0,\delta ]$. 
Then
$(H v_1 )' (t) =H v_2 (t) >0$ for all $t\in (0,\delta ]$. It follows that $v_1 (t) >0 $ for all $t\in (0,\delta ]$. It then
follows from Lemma \ref{Positive} that
the solution is not global.

If $v_2 (0)=0$ then we have by Remark \ref{Basic}  that $v_1 ' (0) =0$ and that $v_3 ' (0) >0$. It
follows that there exists $\delta >0$ such that for $t\in (0,\delta )$, $v_2 ' (t) >0$. Then $v_2 (t) >0$ for all $t\in (0,\delta )$. Then as in the previous
paragraph $v_1 (t) >0$ for $t\in (0,\delta )$, and the solution is not global.

Assume now that  $v_0 (0) <1$. Then there exists $\delta >0$ such that on the interval $(0,\delta]$ we have that
$v_0 <1$. Then $(Hv_3 )' (t) = H\beta (v_0^q -v_0 )(t) <0$ for all $t \in (0,\delta]$. It follows that 
$v_3 (t) <0$ for all $t\in (0,\delta ]$.

If $v_2 (0) <0$ then we can assume that $v_2 (t) <0$ for all $t\in (0,\delta ]$. Then
$(H v_1 )' (t) =H v_2 (t) <0$ for all $t\in (0,\delta ]$. It follows that $v_1 (t) <0 $ for all $t\in (0,\delta ]$
and it follows from Lemma \ref{Negative} that the solution is not global.

If $v_2 (0)=0$ then we have by the Remark \ref{Basic} that $v_1 ' (0) =0$ and that $v_3 ' (0) <0$. It
follows that for $t$ small, $v_2 ' (t) <0$. Then for $t$ small $v_2 (t) <0$. Then as in the previous
paragraph $v_1 (t) <0$ for $t$ small and the solution is not global from Lemma \ref{Negative}..

\end{proof}

\begin{lemma}\label{Negative2} Let $(v_0 , v_1 , v_2 , v_3 )$ be a solution of the system, $v_0 \geq 0$. Assume that for some 
$t_0 \in (0,D)$ we have that $v_0 (t_0 ) \leq 1$, $v_1 (t_0 )=0$ and there exists $\delta >0$ such that
$v_1 (t) <0$ if $0<|t-t_0 |< \delta$. Then the solution is not global.

\end{lemma}

\begin{proof}Note that $v_1$ has a local maximum at $t_0$. Therefore $v_1 ' (t_0 ) =v_0 '' (t_0 )=0$ and
$v_1 '' (t_0 ) = v_0 ''' (t_0 ) \leq 0$. Then we have that $v_2 (t_0 )=0$ and $v_3 (t_0 ) \leq 0$.

If $v_3 (t_0 ) <0$ then there exists $\varepsilon \in (0, \delta )$ such that $v_3 (t) <0$ if $t \in [t_0 , t_0 +\varepsilon )$.
Then we can also assume that $v_2 ' (t) <0$ if $t \in [t_0 , t_0 +\varepsilon )$. Then $v_2  (t) <0$ if $t \in [t_0 , t_0 +\varepsilon )$. We also
have $v_0 (t) <1$ if $t\in (t_0 , t_0 +\varepsilon )$  and $v_1 (t)<0$ if $t\in (t_0 , t_0 +\varepsilon )$. Therefore 
it follows from Lemma \ref{Negative} that the solution is not global. 

Assume that $v_3 (t_0 )=0$. If $v_0 (t_0 ) =1$  it would follow from the uniqueness of solutions that $v_0 \equiv 1$, but 
this is a contradiction since we are assuming the $v_1 <0$ at some points.
Assume then that $v_0 (t_0 )<1$. It follows that $v_3 ' (t_0 )<0$. Then we can assume, by taking $\delta$ smaller, 
that
$v_3 (t) <0$ if $t\in (t_0 , t_0 + \delta )$. Then it follows that $v_2 '(t) <0$ if $t\in (t_0 , t_0 + \delta )$. Therefore $v_2 (t) <0$ if $t\in (t_0 , t_0 + \delta )$. 
We also have that  $v_0 (t)<1$ and $v_1 (t)<0$ if $t\in (t_0 , t_0 + \delta )$. Therefore it follows again from
Lemma \ref{Negative} that the solution is not global.

\end{proof}

The following lemma is proved in the same way, only changing the direction of the inequalities, so we
do not include the proof:

\begin{lemma}\label{Positive2} Let $(v_0 , v_1 , v_2 , v_3 )$ be a solution of the system. Assume that for some 
$t_0 \in (0,D)$ we have that $v_0 (t_0 ) \geq 1$, $v_1 (t_0 )=0$ and there exists $\delta >0$ such that
$v_1 (t) >0$ if $0<|t-t_0 |< \delta$. Then the solution is not global.

\end{lemma}





\begin{lemma}\label{Triple}
Let $(v_0 , v_1 , v_2 , v_3 )$ be a non-constant, non-negative, solution of the system. Assume that for some 
$t_0 \in (0,D)$ we have that $v_1 (t_0 ) = v_2 (t_0 ) = v_3 (t_0 )=0$. Then it is not
a global solution.

\end{lemma}

\begin{proof} If $v_0 (t_0 )=1$ then $v_0$ is the constant solution 1. 

If $v_0 (t_0 )>1$ then it follows
that $v_3 ' (t_0 ) >0$. Then there exists $\delta >0$ such that  $v_3 (t)>0$ for all $t \in (t_0 ,  t_0 + \delta )$. We have
that
$v_1 ' (t_0 ) =0$. We have that $v_2 ' (t_0 )=0$ but $v_2 '' (t_0 )  = v_3 ' (t_0 )>0$. 
It follows that, after choosing $\delta$ smaller if necessary, we can assume that $v_2 ' (t) >0$ for $t \in (t_0  , t_0 + \delta )$. 
Then $v_2 (t) >0$
for all $t \in (t_0 , t_0 + \delta )$. Then $(Hv_1 ) '  (t) =H v_2 (t) >0$ for $t \in (t_0 , t_0 + \delta )$.
Then $v_1 (t) >0$ for $t\in (t_0 , t_0 + \delta )$. Then it follows from Lemma \ref{Positive} that the solution is not global.

If $v_0 (t_0 )<1$ then it follows
that $v_3 ' (t_0 ) <0$. Then for some $\delta >0$ we have that $v_3 (t) <0$ for all $t \in (t_0 ,  t_0 + \delta )$. We have that
$v_1 ' (t_0 ) = v_2 ' (t_0 ) =0$. Then we have that $v_2 '' (t_0 ) <0$ and we can assume that  $v_2 ' (t) <0$ for $t \in (t_0 , t_0 + \delta )$. 
Then $v_2 (t) <0$
for $t \in (t_0 , t_0 + \delta )$. Then $(Hv_1 ) '  (t) =H v_2 (t) <0$ for $t \in (t_0 , t_0 + \delta )$.
Then $v_1 (t) <0$ for $t\in (t_0 , t_0 + \delta )$. Then it follows from Lemma \ref{Negative} that the solution is
not global.

\end{proof}

We are ready to prove that, if $\alpha, \beta >0$, and $\alpha^2 >4 \beta$, critical points of non-constant, positive solutions of Equation \ref{PaneitzODE}
are non-degenerate:

\begin{theorem}\label{Nondegenerate} If $u$ is a positive non-constant  solution of \ref{PaneitzODE} 
then $u(0) \neq 1$, $u(D) \neq 1$, and if $t \in [0,D]$ is a critical point of $u$ then
$u'' (t) \neq 0$.
\end{theorem}

\begin{proof} Recall that under the hypothesis for $h$ we have that if $u$ is a solution of \ref{PaneitzODE} then
$v(t)=u(D-t)$ is also a solution. 

It follows from Lemma \ref{One} that $u(0) \neq 1$. Applying the same lemma to $v$ we also
have that $u(D) \neq 1$.

Let $(v_0 =u,v_1 , v_2, v_3 )$ be the corresponding solution of the system.

Let $t_0 \in [0,D]$ be a critical point of $u =v_0$. Then $v_1 (t_0 ) = u' (t_0 ) =0$. If $t_0 =0$ then
$v_2 (0) = (1+k) v_1 ' (0) = (1+k) u''(0)$. Then it follows from Lemma \ref{AtZero}  that $u'' (0) \neq 0$.
Applying the result to $v$ we see that $u'' (D) \neq 0$. Therefore we can assume that $t_0 \in (0,D)$. 
Note that $v_2 (t_0 ) = v_1 ' (t_0 ) = u'' (t_0 )$. Assume that $v_2 (t_0 )=0$. It follows from 
Lemma \ref{Triple} that $v_3 (t_0 )\neq 0$. Then $v_1 '' (t_0 ) = v_2 ' (t_0 )= v_3 (t_0 ) \neq 0$.
Then $v_1$ has a local extremum at $t_0$ and there exists $\delta >0$ such that if $0<|t-t_0 |< \delta$
then $v_1 (t)$ has a fixed sign. If $v_0 (t_0 ) \geq 1$ and  the sign is positive then we reach a contradiction
from Lemma \ref{Positive2}. If $v_0 (t_0 ) \geq 1$ and the sign is negative we reach a contradiction
applying Lemma \ref{Positive2} to $v$.  If $v_0 (t_0 ) \leq 1$ and  the sign is negative then we reach a contradiction
from Lemma \ref{Negative2}. If $v_0 (t_0 ) \leq 1$ and the sign is positive we reach a contradiction
applying Lemma \ref{Negative2} to $v$.

\end{proof}

\begin{remark}Consider a sequence of functions $u_k$ defined in a closed interval $[0,C]$,
$u_k \in C^{4} [0,C]$,
which converge in the $C^3$ topology to a function $u \in C^4 [0,C]$. 

Assume that for some fixed positive
integer $n$ the functions $u_k$ have exactly $n$ critical points, which include $0$ and $C$. 
Let $0=t_1 (k) < t_2 (k) <...<t_n (k)=C$ be
the n critical points of $u_k$. Assume also that $u_k '' (t_i (k) ) \neq 0$ for each 
$i=1,..., n$. Finally assume that for each $i$ the sequence $t_i (k)$ is convergent
and let $t_i$ be the limit, $0=t_1 
\leq t_2 \leq ... \leq t_n =C$. Note that $u' (t_i ) =0$. 

Note that there exists  $t \in  (t_i (k) , t_{i+1} (k))$ 
such that $u_k '' (t) =0$.  It then follows that if $t_i = t_{1+1}$ then $u '' (t_i )=0$. 
Then if we assume that for any critical point $t$ of $u$ we have that $u'' (t) \neq 0$ then
$u$ must have exactly $n$ critical points.

\end{remark}

\begin{theorem}\label{Close}
Let $\alpha _k , \beta_k$ be positive constants satisfying $\alpha_k^2 \geq 4 \beta_k$. Assume that
the sequences are convergent, $\alpha_k \rightarrow \alpha$, $\beta_k \rightarrow \beta$, $\alpha,
\beta >0$. Let $u_k \in C^5 [0,D]$ be a positive solution of

\begin{equation}
\varphi  '''' + 2h \varphi  ''' +(2h' +h^2 -\alpha_k ) \varphi '' +(h''+hh ' -\alpha_k  h )\varphi  ' + \beta_k (\varphi - \varphi^ q) =0.
\end{equation}

\noindent
Assume that for all $k$, $u_k$ has exactly $n$ critical points and that the sequence $u_k$ converges
($C^4$) to $u\in C^5 [0,D]$. Then $u$ is a positive solution of the limit equation

\begin{equation}
\varphi  '''' + 2h \varphi  ''' +(2h' +h^2 -\alpha ) \varphi '' +(h''+hh ' -\alpha h )\varphi  ' + \beta (\varphi - \varphi^ q) =0.
\end{equation}

\noindent
and it is either equal to the constant solution 1 or has  exactly $n$ critical points

\end{theorem}

\begin{proof} Clearly $u$ is a non-negative solution of the limit equation. Since $u_k$ are 
non-constant solutions it follows that the maximum of $u_k$ is greater than 1. It follows
that $u$ cannot be the constant solution 0. Then it follows from Lemma
\ref{MaximumPrinciple} that $u$ is positive. If $u$ is not the constant solution 1, then it
follows from Lemma \ref{Nondegenerate} that all the critical points of $u_k$ and $u$ are
nondegenerate. Then it follows from the previous remark that $u$ has exactly $n$ critical
points.
\end{proof}

\section{Positive solutions of Equation \ref{Paneitz} with $\alpha , \beta$ close to 0}

We consider the Paneitz-type operator $$P_{(\alpha,\beta)}(u):=\Delta^2u-\alpha\Delta u+\beta u,$$
where $\alpha,\beta>0$.
In this section we shall study positive solutions  of the Paneitz-type equation $ P_{(\alpha,\beta)}u=\beta u^q$
when  the parameters $\alpha$ and $\beta$ are close to zero. For convenience we will at some points renormalize the 
equation as $ P_{(\alpha,\beta)}u= u^q$.

The proof of the next lemma follows the argument in \cite[Theorem 2.2]{Licois}.

\begin{lemma}\label{uigoingtozero} Let $u$ be a  positive function solving 
$ P_{(\alpha ,\beta )}u = u^q$. Assume that 

\begin{equation}\label{smallnorm}
q\|u \|^{q-1}_{\infty}<\mu_1(0),
\end{equation}

where $\mu_1(0)$ is the first positive eigenvalue of the Bilaplacian :

$$\mu_1(0)= \inf_{ \{ u \neq 0 :\int u=0 \} }\dfrac{\int (\Delta u)^2}{\int u^2} .$$

Then $u$ is constant.
\end{lemma}

\begin{proof} We have that 

$$\mu_1(0)= \inf_{\int u=0}\dfrac{\int (\Delta u)^2}{\int u^2} \leq \mu_1(\alpha )=\inf_{\int u=0}\dfrac{\int (\Delta u)^2+\alpha \int |\nabla u|^2}{\int u^2}.$$

Now consider the average $\overline{u}$ of the positive solution $u$, defined by 
$$\overline{u}:=\dfrac{\int_M u}{vol(M,g)}.$$

We get 
\begin{align*}
\mu_1(0)\int(u -\overline{u})^2\leq& \int (\Delta (u -\overline{u}))^2+\alpha \int|\nabla (u -\overline{u})|^2\\
&=\int (\Delta^2(u -\overline{u})-\alpha \Delta(u -\overline{u }))   \    (u-\overline{u})\\
&=\int (\Delta^2u -\alpha  \Delta u )(u -\overline{u})\\
&=\int -\beta u (u -\overline{u} )+\int u^q (u -\overline{u})\\
&=-\beta \int (u -\overline{u}) (u -\overline{u} )+\int (u^q-\overline{u}^q)(u -\overline{u}).
\end{align*}
In last equality we use that $$\int \overline{u} (u -\overline{u})=\int \overline{u}^q(u -\overline{u} )=0.$$

On the other hand, applying the mean value theorem to the function $T(x)=x^q$ we obtain that 
$ (u^q-\overline{u}^q)(u -\overline{u}) \leq q\|u  \|_{\infty}^{q-1}(u -\overline{u}) (u -\overline{u} )$. Then

$$\int (u^q-\overline{u}^q)(u-\overline{u})\leq      q\|u_i\|_{\infty}^{q-1}\int(u  -\overline{u })^2.$$

Combining the previous inequalities, we have
$$\mu_1(0)\int (u -\overline{u })^2\leq -\beta \int(u -\overline{u})^2+q\|u \|_{\infty}^{q-1}\int(u -\overline{u })^2.$$
By condition \ref{smallnorm} $$\int(u -\overline{u} )^2=0.$$ 
Therefore $u$ is constant.

\end{proof}

\begin{theorem}\label{CloseToZero}

Let $q<p^*:=(n+4)/(n-4)$.

There exist $\alpha^* >0$ such that if $\alpha\in (0,\alpha^*)$ and $\beta >0$ with $\alpha^2/4-\beta>0$ then the only positive solution of

\begin{equation}\label{eq}
    P_{(\alpha,\beta)}u=\beta u^q
\end{equation}

is $u=1$.
\end{theorem}

\begin{proof} Assume that there exist positive sequences  $\alpha_i,\beta_i$ with $\alpha_i^2/4-\beta_i>0$ such that
\begin{itemize}
    \item $\alpha_i,\beta_i\rightarrow 0$;
    \item there exists a positive non-constant solution $u_i$  ($\neq \beta_i^{1/(q-1)}$) of $$P_{\alpha_i,\beta_i}u_i=u^q_i .$$
\end{itemize}

Note that the previous lemma gives a positive lower bound for $\|u_i\|_{\infty}$. Therefore 
we have two possibilities for the sequence $\|u_i\|_{\infty}$. Either there exists a subsequence  such that
$\lim_{k\rightarrow \infty} \|u_{i_k}\|_{\infty} = a \in (0,\infty )$  or $\lim_{i \rightarrow \infty } \|u_i\|_{\infty} =\infty$.

Assume first  that $\|u_i\|_{\infty}\rightarrow a \in(0,\infty )$.
Let  $w_i:=u_i/\|u_i\|_{\infty}$, which verifies that $\|w_i\|_{\infty} =1$ and

$$P_{(\alpha_i,\beta_i)}w_i=\|u_i\|^{q-1}_{\infty}w_i^q .$$

Since $\alpha_i^2/4-\beta_i>0$, we can split the operator $P_{(\alpha_i,\beta_i)}$ into  the following form
$$(-\Delta+c_1 (i) )\circ(-\Delta+c_2 (i) ),$$
where $$c_1(i)=\alpha_i/2+\sqrt{\alpha_i^2/4-\beta_i}  >0 \quad\text{and}\quad c_2(i)=\alpha_i/2-\sqrt{\alpha_i^2/4-\beta_i} >0.$$

Therefore we have that $(-\Delta+c_1 (i) ) \ (-\Delta+c_2 (i) ) (w_i ) = \|u_i\|^{q-1}_{\infty}w_i^q.$ 
Then we can apply regularity theory
for elliptic equations to obtain the existence of a subsequence such that $\lim w_i = w \in C^{4}$, with  $\|w \|_{\infty} =1$,
$w\geq 0$,  and 

$$\Delta^2w=a^{q-1}w^q.$$

In particular this implies that $\int w^q=0$ and we conclude that $w=0$, which contradicts that $\|w \|_{\infty} =1$.

\medskip

We can therefore assume that  $\|u_i\|_{\infty}\rightarrow\infty$. Note that up to this point we have only used  that $q>1$.
Assuming that $q$ is subcritical, $q<p^*$ we can apply the blow up technique. 

Pick points $x_i \in M$ so that
$u_i (x_i )= \|u_i\|_{\infty} = \lambda_i \rightarrow\infty$. Also let $r$ be a positive number, less than the injectivity radius of $(M,g)$. 
Then define a sequence of functions $w_i$ on $B\left( 0,\lambda_i^{\frac{4}{q-1}}  r \right) $, the Euclidean ball with center at the origin
and radius $\lambda_i^{\frac{4}{q-1}}  r$,   by $$w_i(x):=\frac{1}{\lambda_i} u_i  \left(  exp_{x_i}(\lambda_i^{-\frac{q-1}{4}} x)  \right)  .$$

The functions $w_i$ verify that $w_i (0)=1 = \|  w_i \|_{\infty} $. Let $g$ still denote the metric on $B(0,r))$ which is the pull-back of $g$ by the
exponential map and consider the metric $g_0$ on  $B\left( 0,\lambda_i^{\frac{4}{q-1}}  r \right) $  given by $g_0 (x) = g \left( \lambda_i^{\frac{-4}{q-1}} \ x \right) $.
Denote by $P^{g_0 }$ the Paneitz operator with respect to the metric $g_0$ and let $\widetilde{\alpha}_i =\alpha_i \lambda_i^{\frac{1-q}{2}} $,
$\widetilde{\beta}_i = \beta_i\lambda_i^{(1-q)} $. By a direct computation we have that $w_i$ verifies 
the equation 

$$P^{g_0}_{(\widetilde{\alpha}_i  ,
\widetilde{\beta}_i ) }w_i=w_i^q.$$

Note that $\widetilde{\alpha}_i^2  \geq 4 \widetilde{\beta}_i$ and then as in Remark \ref{Decomposition}
we can split the operator $P_{(\widetilde{\alpha}_i  ,\widetilde{\beta}_i  )}$ and $w_i$ satisfies an equation
 $$(-\Delta_{g_0} +c_1(i))\circ(-\Delta_{g_0} +c_2(i)) w_i = w_i^q , $$

\noindent
with $c_1 (i), c_2 (i) \rightarrow 0$. 

Note also that the metric $g_0$ depends on $i$ and $C^k$-converges to the Euclidean metric on any compact set $K$.
It then follows from elliptic theory that the sequence $w_i$ converges to a nonnegative  function $W\in C^4(\mathbb{R}^n)$ which is a
solution of the problem

\begin{align*}
\Delta^2W&=W^q\quad\text{in} º \ \mathbb{R}^n,\\
W(0)&=1.
\end{align*}
Since $q<p^*$,  this is a contradiction with the results  in \cite{Lin, Wei2}.

\end{proof}

For any $0<\varepsilon < K $ let $S=S_{\varepsilon , K} = \{(u,\alpha , \beta ) : u$ is a positive solution of \ref{Paneitz}
and $\alpha , \beta  \in [\varepsilon , K]$ , with $\alpha^2 \geq 4\beta  \} $.

\begin{theorem}\label{Compact}
If $q<p^*$ then $S$ is compact.

\end{theorem}

\begin{proof} Let $w_i$ be a sequence of functions in $S$, $w_i$ is a positive solution of Equation 
\ref{Paneitz} with $\alpha_i , \beta_i \in [\varepsilon , K]$ and $\alpha_i^2 \geq 4\beta_i   $.  Note that this implies that
$\| w_i \|_{\infty} \geq 1$.
If $\| w_i \|_{\infty} \rightarrow \infty$ then  we can 
apply the blow-up technique as in the  proof of the previous theorem to reach a contradiction. Then we can assume
that the sequence is bounded in $L^{\infty}$. As in \ref{Decomposition} the corresponding operator
$P$ splits as the composition of two second order
elliptic operators.  By passing to a subsequence if necessary we can assume that
$\alpha_i \rightarrow \alpha$ and $\beta_i \rightarrow \beta$, for some 
$\alpha , \beta  \in [\varepsilon , K]$ with  $\alpha^2 \geq 4\beta $.
We then see using regular elliptic theory that the sequence is bounded in $C^{4,\theta}$ and therefore
it converges (up to a subsequence) to a nonnegative function $W$  which is a solution of Equation\ref{Paneitz} 

$$P_{\alpha , \beta} W = (-\Delta +c )(-\Delta +d) W = \beta W^q .$$

If we denote by $U=-\Delta W + d W$ then the equation implies that $-\Delta U +cU \geq 0$. This implies that
$U \geq 0$. And then by the maximum principle $W\equiv 0$ or $W$ is strictly positive. Since 
$\| W \|_{\infty} \geq 1$ we have that $W$ is strictly positive, and therefore $W\in S$. 

\end{proof}







\section{Bifurcation}

We consider  a closed Riemannian manifold $(M,g)$ with a proper isoparametric function 
$F:M \rightarrow [t_0 ,t_1]$. As in Section 2, we consider the focal submanifolds $M_0 = F^{-1} (t_0 )$
and $M_1 = F^{-1} (t_1 )$. 
We let ${\bf d}:M \rightarrow [0,D]$ be the distance function to $M_0$, $D=d_g (M_0 , M_1 )$.
We let $\mathcal{B}_4 = \{ \phi \in C^{4,\alpha} ([0,D] ): \phi '(0)  =  \phi ''' (0)=0 = \phi '(D)  = \phi '''(D)\}$,
which is identified with the $C^{4,\alpha}$ functions on $M$ which are $F$-invariant. We let $h(t)$ be the mean curvature of
${\bf d}^{-1} (t)$.

Consider  smooth, positive, increasing  functions ${\bf {a}} , {\bf b} : (0,\infty ) \rightarrow (0,\infty )$. We  will assume that
${\bf a}^2 (t) >4{\bf b} (t)$ for all $t\in (0,\infty )$, 
$\lim_{t \rightarrow 0} {\bf a} (t) =\lim_{t \rightarrow 0} {\bf b} (t) =0$ and
$\lim_{t \rightarrow \infty } {\bf a} (t)  =    \lim_{t \rightarrow \infty} {\bf  b} (t) =\infty$.

Then we define the map  $H: \mathcal{B}_4  \times (0,\infty ) \rightarrow  C^{0,\alpha } ([0,D]$, by

$$H(u,t)=  u '''' + 2h u  ''' +(2h' +h^2 -{\bf a} (t)  ) u '' +(h''+hh ' -{\bf a} (t) h ) u ' + {\bf b} (t)  (u - u^ q) .$$

We look for solutions of the equation $H(u,  t )=0$ with $u>0$. We call $(1, t)$, $t \in (0,\infty )$, the
family of {\it trivial solutions}. A solution $(u,t)$ with $u\neq 1$ is called {\it nontrivial}.

\medskip

The linearization of the equation $H(u,t)=0$ at the constant solution $(1,t)$  is given by 

\begin{equation}
D_u  H (1,t) [v] = v '''' + 2h v  ''' +(2h' +h^2 -{\bf a} (t)  ) v '' +(h''+hh ' -{\bf a} (t) h ) v ' + {\bf b} (t)  (1-q) v =0.
\end{equation}

Let  $L:\mathcal{B}_2 \rightarrow  C^{0,\alpha } ([0,D]$ be given by
$L(u) = u'' +h u'$. The linearized equation at $(1,t)$  is $L^2 (v) -{\bf a}(t)  L(v) +{\bf b}(t) (1-q)v=0$, which is
the eigenvalue equation for $L^2 -{\bf a} (t)  L$. 

The space of solutions of the eigenvalue problem $L(v) = \lambda v$ has at most dimension 1, since 
it is a second order equation, and solutions must verify $v'(0)=0$. Sturm Liouville theory says that $L$ has 
infinite eigenvalues
$0=\lambda_0 > \lambda_1 > \lambda_2 ...\rightarrow  -\infty$. Moreover, if $\varphi_0 =1 , \varphi_1 , \varphi_2 , ...$ are corresponding
nonzero eigenfunctions, $L(\varphi_i ) = \lambda_i \varphi_i $, then $\varphi_i$ has exactly $i$ zeroes in $(0,D)$.
Then we have that $\{ \varphi_i \}$ is a complete set of eigenfunctions for $L^2 -{\bf a} (t) L$, with corresponding
eigenvalues $\mu_i =\lambda_i^2 -{\bf a}(t)  \lambda_i$. 

\medskip

Therefore the linearized equation (at $(1,t)$) has a solution if and only if for some $i$ we have that
$\lambda_i^2 -{\bf a} (t) \lambda_i = {\bf b} (t) (q-1) >0$. We must have
$\lambda_i = (1/2) ({\bf a} (t) \pm \sqrt{ {\bf a}^2 (t) + 4 {\bf b} (t) (q-1)}$. 
And since $\lambda_i $ must be nonpositive we must 
have  $\lambda_i  = (1/2) ({\bf a} (t)  - \sqrt{ {\bf a}^2 (t) + 4 {\bf b} (t) (q-1)} = \phi (t)$.

To apply bicurcation theory we would like that 
$\phi$ is a decreasing function, $\lim_{t\rightarrow 0} \phi (t) =0$, and $\lim_{t\rightarrow \infty} \phi (t) = -\infty$.
Under these conditions for each $i\geq 1$ there exists a unique $t_i   =  t_i ({\bf a}, {\bf b})>0$ such that $\lambda_i = \phi (t_i )$.

We will study bifurcation from the family of trivial solutions $t \mapsto (1,t)$. The point $(1,t^* )$ is called a 
{\it bifurcation point} for the family if it is in the closure of the space of nontrivial solutions.

Note that

$$ D^2_{u,t} H (1,t_i ) [\varphi_i ]= -{\bf a}' (t_i ) L(\varphi_i ) + {\bf b}' (t_i ) (1-q) \varphi_i .$$

Then 

$$ \int  D^2_{u,t} H (1,t_i ) [\varphi_i ]  \ \varphi_i  = \left( -{\bf a}' (t_i )  \lambda_i  + {\bf b}' (t_i ) (1-q) \right) \int \varphi_i^2 .$$

Since $T := D_u H (1, t_i)[ \cdot]$  is self-adjoint and $ker(T) = span( \varphi_i )$ the last computation implies that
if $-{\bf a}' (t_i )  \lambda_i  + {\bf b}' (t_i ) (1-q) \neq 0$ then

$$D^2_{u,t} H (1, t_i )[\varphi_i  ] \notin  Range(T).$$

Then the following result follows from the previous discussion and the classical Local Bifurcation Theorem for simple eigenvalues of M. G. Crandall and P.  Rabinowitz 
\cite{Crandall}:

\begin{proposition}\label{LocalBifurcation} Assume that ${\bf a}^2 (t) >4{\bf b} (t)$ for all $t\in (0,\infty )$, 
$\lim_{t \rightarrow 0} {\bf a} (t) =\lim_{t \rightarrow 0} {\bf b} (t) =0$ and
$\lim_{t \rightarrow \infty } {\bf a} (t)  =    \lim_{t \rightarrow \infty} {\bf  b} (t) =\infty$. Assume also that 
$\phi (t) =(1/2) ({\bf a} (t)  - \sqrt{ {\bf a}^2 (t) + 4 {\bf b} (t) (q-1)}$ is decreasing,
$\lim_{t\rightarrow 0} \phi (t) =0$, and $\lim_{t\rightarrow \infty} \phi (t) = -\infty$.
Let $0<t_1 < t_2 <...\rightarrow \infty$ be the values such that $\lambda_i = \phi (t_i )$. 
Assume that for each positive integer $i$, $-{\bf a}' (t_i ) \lambda_i + {\bf b}' (t_i ) (1-q) \neq 0$.
Then $\{ (1, t_i ) \}$ is the set  of
bifurcation points of the equation $H(u,t)=0$ at the family of trivial solutions $\{ (1,t) \}$. These bifurcation points are simple, $\varphi_i$ is a solution of the
linearized equation at $(1, t_i )$. The set of nontrivial solutions around this bifurcation point is given by a path
$s \mapsto (u_i (s), t_i (s))$ with $u_i (s) =1 +s \varphi_i + o(s^2 )$. For each $i$ the eigenfunction
$\varphi_i$ has exactly $i+1$ nondegenerate critical points (counting 0 and $D$), and the same holds
for $u_i (s)$ if $s$ is close enough to 0.

\end{proposition}

The proof is a direct application of the classical theorem of Crandall and Rabinowitz.

\bigskip

We have now all the elements needed to prove our main global bifurcation result:

\begin{theorem} Assume the conditions of the previous proposition.
For each $t > t_i$ the equation $H(u,t)=0$ has at least $i$ distinct solutions with $u$ positive and non-constant, at least one such that $u$ has 
exactly $k$ critical points
for each  $2\leq k \leq i+1$.

\end{theorem}

\begin{proof} Consider the paths of local nontrivial solutions bifurcating at $(1, t_i )$, from Proposition \ref{LocalBifurcation}. Let
$S_i$ be the connected component, in the closure of the space of non-trivial solutions of the equation, containing the path. Let $S_i^0
\subset S_i$ be the subset of solutions which have exactly $i+1$ critical points. Note that non-constant solutions close to the bifurcation 
point $(1,t_i )$ are in $S_i^0$. It follows from Lemma \ref{Nondegenerate} that all the critical points of solutions are
nondegenerate and this implies that $S_i^0$ is open in $S_i$. 
It also follows from Proposition \ref{LocalBifurcation} that if $j\neq i$ then $(1,t_j ) \notin \overline{S_i^0}$.
It follows from Theorem \ref{Close} that 
$S_i^0 \cup \{ (1,t_i ) \}$ is closed in $S_i$ and therefore $S_i = S_i^0 \cup \{ (1,t_i ) \}$. Then all solutions on $S_i$
have exactly $i+1$ critical points. This implies that if $i\neq j$ we have that $S_i \cap S_j =\emptyset$.
Then we know from the Global Bifurcation Theorem of P. Rabinowitz \cite{Rabinowitz} that $S_i$ is not compact. 
It follows from Theorem \ref{CloseToZero}, that there exists $\varepsilon >0$ such that $S_i \subset C^{4,\alpha} [0,D]
 \times [\varepsilon , \infty )$. If there exists $t>t_i$ such that $S_i \cap C^{4,\alpha} [0,D] \times \{ t \} =
\emptyset$ then we would have that $S_i \subset  C^{4,\alpha} [0,D]
 \times [\varepsilon , t )$. Then it would follow from Theorem \ref{Compact} that $S_i$ is compact. Therefore
$S_i$ intersects $ C^{4,\alpha} [0,D] \times \{ t \} $ for each $t>t_i$ and the theorem follows.

\end{proof}

\begin{remark}
If we consider the Equation \ref{PaneitzODE} for some fixed positive values of $\alpha , \beta$ which verify $\alpha^2 >4\beta$, then we
can consider functions  ${\bf {a}} , {\bf b} : (0,\infty ) \rightarrow (0,\infty )$ satisfying the conditions in this section, and such that for some 
$s_0 \in (0,\infty )$,
 ${\bf {a}} (s_0 )= \alpha $ and  $ {\bf b} (s_0 )=\beta$. 

Then we can apply the previous theorem and if $\alpha , \beta$ verify that
 $(1/2) (\alpha  - \sqrt{\alpha^2  + 4\beta  (q-1)} )< -\lambda_i$ then Equation \ref{PaneitzODE} has at least $i$ nontrivial positive solutions.

For instance we can consider ${\bf a} (t) = t$ and ${\bf b}(t) = ct^2$, with $c\in (0,1/4)$. These clearly verify the  conditions
${\bf a}^2 (t) >4{\bf b} (t)$ for all $T\in (0,\infty )$, 
$\lim_{t \rightarrow 0} {\bf a} (t) =\lim_{t \rightarrow 0} {\bf b} (t) =0$ and
$\lim_{t \rightarrow \infty } {\bf a} (t)  =    \lim_{t \rightarrow \infty} {\bf  b} (t) =\infty$.
And $\phi (t) = t/2 (1-\sqrt{1+4c(q-1)})$ is a linear decreasing function,
$\lim_{t\rightarrow 0} \phi (t) =0$, and $\lim_{t\rightarrow \infty} \phi (t) = -\infty$. In this case we can compute explicitly $t_i =
2\lambda_i  \left(  1-\sqrt{1+4c(q-1)} \right)^{-1}$. 
Then 

$$-{\bf a}' (t_i ) \lambda_i + {\bf b}' (t_i ) (1-q) =\lambda_i  \left( -1+4c (1-q) \left( 1-\sqrt{1+4c(q-1)} \right)^{-1} \right) .$$

By a direct computation we see that this expression does not vanish (when $c>0$, $q>1$). Therefore for any
given $\alpha$, $\beta$ satisfying $\alpha^2 > 4\beta$ we can choose $c \in (0,1/4)$ so that $\beta = c\alpha^2$, and we
can apply the theorem.

\end{remark}

\section{Appendix}

Let 

$$\alpha =\frac{N^2-4N+8}{2(N-1)(N-2)} (n\Lambda_0 + m\Lambda_1 ) -  \frac{4}{N-2} \Lambda_0$$

\noindent  
and 

$$\beta =\frac{N-4}{2}  \left(   \frac{-2}{(N-2)^2} \left( n\Lambda_0^2  + m\Lambda_1^2  \right)  + 
\frac{N^3 -4N^2 +16N-16}{8(N-1)^2 (N-2)^2} \left( n\Lambda_0 + m \Lambda_1 \right)^2   \right) .$$

We assume that $m,n\geq 3$, $N=n+m \geq 6$, and that $\Lambda_1 \geq \Lambda_0 >0$. 
We prove here the claim given in the introduction that  under these 
conditions $\alpha^2 -4\beta >0$.

\medskip

We consider $n,m$ fixed and consider $\alpha =\alpha (\Lambda_0 , \Lambda_1 )$ and $\beta = \beta (\Lambda_0 , \Lambda_1 )$ 
as functions of $\Lambda_0$ and
 $\Lambda_1$. Let $h(t) = \alpha (\Lambda_0 , t)^2 -4\beta (\Lambda_0 , t)$. We want to
prove that $h(t) >0$ if $t\geq  \Lambda_0$.

We first compute  

$$h(\Lambda_0 )=  \Lambda_0^2  \left(           \frac{N^2-4N+8}{2(N-1)(N-2)} N -  \frac{4}{N-2}            \right)^2
-2 \Lambda_0^2 (N-4) \left(
 \frac{-2}{(N-2)^2}  N + 
\frac{N^3 -4N^2 +16N-16}{8(N-1)^2 (N-2)^2} N^2 \right) $$

$$=\frac{\Lambda_0^2}{4(N-1)^2 (N-2)^2}  \left(  (N^3-4N^2 +8N -8(N-1) )^2 -(N-4)( N^5 -4N^4  +16N^2 -16N ) \right) .$$

$$=\frac{\Lambda_0^2}{4(N-1)^2 (N-2)^2}  ( (N^3 -4N^2 +8)^2 -  N^6  +8N^5 -16N^4  -16N^3 +80N^2  -64N) $$

$$=\frac{\Lambda_0^2}{4(N-1)^2 (N-2)^2} (16N^2 -64N +64) >0 . $$

Also:

$$h' (\Lambda_0 )= 2 m \Lambda_0 \left (   \frac{N^2-4N+8}{2(N-1)(N-2)} N   -  \frac{4}{N-2}  \right)   \frac{N^2-4N+8}{2(N-1)(N-2)} $$

$$ -2m\Lambda_0(N-4) \left( \frac{-4}{(N-2)^2}   + \frac{N^3 -4N^2 +16N-16}{4(N-1)^2 (N-2)^2} N \right) $$


$$= \frac{m\Lambda_0}{2(N-1)^2 (N-2)^2}  (8N^3  - 40N^2    +48N  ) >0.$$

\medskip

Note that $h(t)$ is a quadratic polynomial on $t$. The second order coefficient is:

$$(1/2) h''(t) = \left(  \frac{N^2-4N+8}{2(N-1)(N-2)} \right)^2 m^2  -2(N-4)  \left( \frac{-2m}{(N-2)^2}  + \frac{N^3 -4N^2 +16N-16}{8(N-1)^2 (N-2)^2} m^2  \right)  $$


$$=\frac{m}{4(N-1)^2 (N-2)^2} ( 16Nm  +16(N-4)(N-1)^2 )  >0.$$

The three inequalities clearly imply the claim.

\end{document}